\documentclass[10pt]{amsart}

\usepackage{amsmath}
  \usepackage{paralist}
  \usepackage{graphics} %% add this and next lines if pictures should be in esp format
  \usepackage{epsfig} %For pictures: screened artwork should be set up with an 85 or 100 line screen
\usepackage{graphicx}  \usepackage{epstopdf}%This is to transfer .eps figure to .pdf figure; please compile your paper using PDFLeTex or PDFTeXify.

 \usepackage[colorlinks=true]{hyperref}
\hypersetup{urlcolor=blue, citecolor=red}

\usepackage{comment}

\newtheorem{theorem}{Theorem}[section]

\newtheorem*{main*}{Main Theorem}
\newtheorem{lemma}[theorem]{Lemma}
\newtheorem{proposition}[theorem]{Proposition}

\theoremstyle{definition}
\newtheorem{definition}[theorem]{Definition}
\newtheorem{remark}[theorem]{Remark}

\newtheorem*{TheoremA}{Theorem A}
\newtheorem*{TheoremB}{Theorem B}
\newtheorem*{TheoremC}{Theorem C}
\newtheorem*{TheoremD}{Theorem D}
\newtheorem*{TheoremE}{Theorem E}
\newtheorem*{TheoremF}{Theorem F}
\newtheorem*{TheoremG}{Theorem G}

\newtheorem*{CorollaryA1}{Corollary A.1}
\newtheorem*{CorollaryA2}{Corollary A.2}

\newtheorem*{CorollaryC1}{Corollary C.1}
\newtheorem*{CorollaryC2}{Corollary C.2}

\newtheorem*{CorollaryF1}{Corollary F.1}

\title[Running heading with forty characters or less]
      {Unstable pressure and u-equilibrium states for partially hyperbolic diffeomorphsims}

\def\a{\alpha}
\def\b{\beta}

 \def\e{\varepsilon}

\def\ae{-\text{a.e.}\ }

\def\P{{\mathcal P}}

\def\det{\mathop{\hbox{{\rm det}}}}
\def\diam{\mathop{\hbox{{\rm diam}}}}

\def\loc{{\mathop{\hbox{\footnotesize  \rm loc}}}}

\def\disp{\displaystyle}

\author[]{Huyi Hu, Weisheng Wu and Yujun Zhu}


\address{Department of Mathematics, Southern University of Science and technology of China, Shenzhen, P. R. China\footnote{permenant addrss: Department of Mathematics, Michigan State University, East Lansing, MI 48824, USA, \email{hhu@math.msu.edu}}}
 \email{huhy@sustc.edu.cn}
\address{Department of Applied Mathematics, College of Science, China Agricultural University, Beijing, 100083, P. R. China}
 \email{wuweisheng@cau.edu.cn}

\address{School of Mathematical Sciences, Xiamen University, Xiamen, 361005, P. R. China}
 \email{yjzhu@xmu.edu.cn}

\begin{document}
\maketitle
\markboth{Unstable pressure and u-equilibrium states}
{Huyi Hu, Weisheng Wu and Yujun Zhu}
\renewcommand{\sectionmark}[1]{}

\begin{abstract}
Unstable pressure and u-equilibrium states are introduced and investigated for a partially hyperbolic diffeomorphsim $f$. We define the u-pressure $P^u(f, \varphi)$ of $f$ at a continuous function $\varphi$ via the dynamics of $f$  on local unstable leaves. A variational principle for unstable pressure $P^u(f, \varphi)$,
which states that $P^u(f, \varphi)$ is the supremum of the sum of the unstable entropy and the integral of $\varphi$ taken over all invariant measures, is obtained. U-equilibrium states at which the supremum in the variational principle attains and their relation to Gibbs u-states are studied. Differentiability properties of unstable pressure, such as tangent functionals, Gateaux differentiability and Fr\'{e}chet differentiability and their relations to u-equilibrium states, are also considered.
\end{abstract}

\section{Introduction}

Entropy and pressure are important invariants in the study of
dynamical systems and ergodic theory. Entropies, including
topological entropy and measure-theoretic entropy, are
measurements of  complexity of the orbit structure of the system from different points of view.
As a generalization of entropy, the concept of pressure was introduced by Ruelle \cite{DR} and studied in the general case in Walters \cite{W75}. In fact, the theory of pressure and its related topics, such as Gibbs measures and equilibrium states, are the main constituent components of the
mathematical statistical mechanics.

Let $f$ be a $C^{1+\alpha}$-diffeomorphism on a closed Riemannian manifold $M$ where $\alpha >0$. The well-known entropy formula in \cite{LY2} shows that if $\mu$ is an SRB measure, then the corresponding metric entropy $h_{\mu}(f)$ is the integration of the summation of the positive Lyapunov exponents. It tells us that positive exponents have contribution to the metric entropy. In particular, when $f$ is uniformly hyperbolic, both of the metric entropy and topological entropy are caused by the dynamics on the unstable foliations. However, when $f$ is (uniformly) partially hyperbolic, things become delicate. The presence of the center direction makes the dynamics much more complicated.

Recent years, the entropy theory for partially hyperbolic diffeomorphisms are increasingly investigated. We can see the progress in this research topic in \cite{HSX}, \cite{WZ}, \cite{Yang}, \cite{HHW}, etc.
In particular, for any $C^1$-partially hyperbolic diffeomorphism $f$, Hu, Hua and Wu \cite{HHW} introduce the definitions of unstable metric entropy $h^u_{\mu}(f)$ for  any invariant measure $\mu$ and unstable topological entropy $h_{\text{top}}^u(f)$. Precisely, $h^u_{\mu}(f)$ is defined by using
$H_{\mu}(\bigvee_{i=0}^{n-1}f^{-i}\alpha|\eta)$, where $\alpha$ is a finite measurable partition, and $\eta$ is a measurable
partition subordinate to unstable manifolds
that can be obtained by refining a finite partition into pieces of
unstable leaves; $h_{\text{top}}^u(f)$ is defined by the topological entropy of $f$ on local unstable manifolds. Similar to that in the classical entropy theory, the corresponding versions of Shannon-McMillan-Breiman theorem and local entropy formula for $h^u_{\mu}(f)$, and the variational principle relating $h^u_{\mu}(f)$ and $h_{\text{top}}^u(f)$ are given. The main feature of these unstable entropies is to rule out the complexity caused by the center directions and focus on that caused by the unstable directions. In fact, $h^u_{\mu}(f)$ is equal to $h_{\mu}(f, \xi):=H_{\mu}(\xi|f\xi)$ (where $\xi$ is an increasing partition subordinate to the unstable
leaves) which was introduced by Ledrappier and Young \cite{LY2}. Comparing the above two types of definition for the unstable metric entropy, we can see that the former one is more natural and easy-to-understand than  the latter one.

The main purpose of this paper is to introduce unstable topological pressure $P^u(f, \varphi)$ for a $C^1$-partially hyperbolic diffeomorphism $f: M\to M$ and any continuous function $\varphi$ on $M$, obtain a variational principle for this pressure, and investigate the corresponding so-called u-equilibriums.

Similar to the way by which the unstable entropy is defined in \cite{HHW}, the unstable pressure $P^u(f, \varphi)$ is defined via the information of the potential $\varphi$ as iterating $f$ on local unstable leaves (see Definition \ref{unstableentropy1}).
It is well known that the variational principle for the classical pressure was first given by Ruelle \cite{DR} for the system with the expansiveness and specification assumptions and then was obtained by Walters \cite{W82} for the general case. It shows that
$$
P(f, \varphi)=\sup \Big\{h_\mu(f)+\int_M \varphi d\mu: \mu \in \mathcal{M}_f(M)\Big\}
$$
where $\mathcal{M}_f(M)$ is the set of all $f$-invariant probability measures on $M$. We will combine the elegant method in Walters \cite{W82} and the technique in Hu, Hua and Wu \cite{HHW} to obtain the variational principle for unstable pressure (Theorem A), i.e., the equality as the above in which $P(f, \varphi)$ and $h_\mu(f)$ are replaced by $P^u(f, \varphi)$ and $h^u_\mu(f)$ respectively. In particular, if $\varphi\equiv 0$, then we get the variational principle for unstable entropy (Theorem D of \cite{HHW}).

The measure at which the supremum attains in the variational principle for unstable pressure $P^u(f, \varphi)$ is called a \emph{u-equilibrium state for $f$ at $\varphi$} (see Definition \ref{u-equilibrium}). Some fundamental properties for the set of u-equilibrium states are considered (Theorem B). Among these properties, we show that there always exists a u-equilibrium state for a $C^1$-partially hyperbolic diffeomorphism, in contrast to the case for the classical equilibrium state. This is essentially due to the upper semi-continuity of the unstable entropy map $\mu\mapsto h^u_\mu(f)$. For particular potential $\varphi^u=-\log \|Df|_{E^u}\|$, we relate the u-equilibrium states at $\varphi^u$ to the Gibbs u-states of $f$ (Theorem C). In \cite{W82} and \cite{W92}, some properties about the classical pressure and equilibrium states were investigated. We can consider the corresponding properties for unstable pressure and u-equilibrium states. We show that unstable pressure determines invariant measures (Theorem D) and there is a close relation between u-equilibrium states and tangent functionals (Theorem E). To study the uniqueness of u-equilibrium state, we define two types, Gateaux type and Frech\'{e}t type, of differentiability of unstable pressure of $f$ at $\varphi$, and obtain several properties of them (Theorems F and G).

The paper is organized as follows. In Section \ref{statements}, we give the definitions of unstable pressure and u-equilibrium state, and formulate the main results. We
provide some properties of unstable pressure in Section \ref{u-pressure}. Section \ref{pfThmA} is for the proof of the variational principle of unstable pressure. In Section \ref{uequi} and \ref{determin}, we consider the properties of u-equilibrium states and study how does the unstable pressure determine invariant measures. In Section \ref{Differentiability}, differentiability properties of unstable pressure are investigated.

\section{Definitions and statements of results}\label{statements}

Let $M$ be an $n$-dimensional smooth, connected and compact Riemannian manifold without boundary and $f: M \rightarrow M$ a $C^{1}$-diffeomorphism. $f$ is said to be \emph{partially hyperbolic} (cf. for example \cite{RRU}) if there exists a nontrivial $Df$-invariant splitting $TM= E^s \oplus E^c \oplus E^u$
of the tangent bundle into stable, center, and unstable distributions, such that all unit vectors $v^{\sigma} \in E_x^\sigma$ ($\sigma= c,s,u$) with $x\in M$ satisfy
\begin{equation*}
\|D_xfv^s\| < \|D_xfv^c\| < \|D_xfv^u\|,
\end{equation*}
and
\begin{equation*}
\|D_xf|_{E^s_x}\| <1 \ \ \ \text{\ and\ \ \ \ } \|D_xf^{-1}|_{E^u_x}\| <1,
\end{equation*}
for some suitable Riemannian metric on $M$. The stable distribution $E^s$ and unstable distribution $E^u$ are integrable to the stable and unstable foliations $W^s$ and $W^u$ respectively such that $TW^s=E^s$ and $TW^u=E^u$ (cf. \cite{HPS}).

In this paper we always assume that $f$ is a $C^1$-partially hyperbolic
diffeomorphism of $M$, and $\mu$ is an $f$-invariant probability measure.
The notion of unstable metric entropy of $\mu$ with respect to $f$ is introduced in \cite{HHW}, using
a type of measurable partitions consisting of local unstable leaves
that can be obtained by refining a finite partition into pieces of
unstable leaves.  We recall the construction of such measurable partitions and the definition of unstable metric entropy as follows. To begin with, we recall some standard notations and classical results on measurable partitions.

Let $(X, \mathcal{A}, \nu)$ be a stand probability space. For a partition $\a$ of $X$, let $\a(x)$ denote the element of $\a$
containing $x$.
If $\a$ and $\b$ are two partitions such that $\a(x)\subset \b(x)$
for all $x\in X$, we then write $\a \geq \b$ or $\b\leq \a$.
A partition $\xi$ is \emph{increasing} if $f^{-1}\xi \geq \xi$.
$\alpha\vee\beta:=\{A\cap B: A\in \alpha, B\in \beta\}$ is called the \emph{refinement} of $\alpha$ and $\beta$.
For a partition $\b$, we denote
$\disp\b_m^n=\vee_{i=m}^n f^{-i}\b$.  In particular,
$\disp\b_0^{n-1}=\vee_{i=0}^{n-1} f^{-i}\b$. A partition $\eta$ of $X$ is called \emph{measurable} if there exists a countable set $\{A_n\}_{n\in \mathbb{N}}\subset \mathcal{B}(\eta)$ such that
for almost every pair $C_1,C_2\in \eta$, we can find some $A_n$ which separates them
in the sense that $C_1\subset A_n, C_2\subset X-A_n$, where $\mathcal{B}(\eta)$ is the sub-$\sigma$-algebra of elements of $\mathcal{A}$ which are unions of elements of $\eta$. The \emph{canonical system
of conditional measures for $\nu$ and $\eta$} is a family of probability
measures $\{\nu_x^\eta: x\in X\}$ with $\nu_x^\eta\bigl(\eta(x)\bigr)=1$,
such that for every measurable set $B\subset X$, $x\mapsto \nu_x^{\eta}(B)$
is measurable and
\[
\nu (B)=\int_X\nu_x^{\eta}(B)d\nu(x).
\]
The classical result of Rokhlin (cf. \cite{R}) says that if $\eta$ is a measurable partition, then there exists a system of conditional measures relative to $\eta$. It is essentially unique
in the sense that two such systems coincide in a set of full $\nu$-measure. For measurable partitions $\a$ and $\eta$, let
$$H_\nu(\a|\eta):=-\int_M \log \nu_x^\eta(\alpha(x))d\nu(x)$$
denote the conditional entropy of $\a$ given $\eta$ with respect to $\nu$.

Now consider a $C^1$-partially hyperbolic diffeomorphism $f:M\to M$. Take $\e_0>0$ small.
Let $\P=\P_{\e_0}$ denote the set of finite Borel partitions of $M$
whose elements have diameters smaller than or equal to $\e_0$, that is,
$\diam \a:=\sup\{\diam A: A\in \a\}\le \e_0$.
For each $\b\in \P$ we can define a finer partition $\eta$ such that
$\eta(x)=\b(x)\cap W^u_\loc(x)$ for each $x\in M$, where $W^u_\loc(x)$
denotes the local unstable manifold at $x$ whose size is
greater than the diameter $\e_0$ of $\beta$. Since $W^u$ is a continuous foliation, $\eta$ is a measurable partition with respect to any Borel probability measure on $M$.
Let $\P^u$ denote the set of partitions $\eta$ obtained in this way and \emph{subordinate to unstable manifolds}.
%We denote such partition $\b$ by $\eta^\sharp$, that is, for any
%$\eta^\sharp\in \P$ and $\eta(x)=\eta^\sharp(x)\cap W^u_\loc(x)$ for any
%$x\in M$.
Here a partition $\eta$ of $M$ is said to be subordinate to unstable manifolds of $f$ with respect
to a measure $\mu$ if for $\mu$-almost every $x$,
$\eta(x)\subset W^u(x)$
and contains an open neighborhood of $x$ in $W^u(x)$.
It is clear that if $\a\in \P$ such that
$\mu(\partial \a)=0$ where $\partial \a:=\cup_{A\in \a} \partial A$,
then the corresponding $\eta$ given by $\eta(x)=\a(x)\cap W^u_\loc(x)$
is a partition subordinate to unstable manifolds of $f$.
%We denote by $\P^u$ the set of partitions in $\P^u_0$ that are
%subordinate to unstable manifolds of $f$.

\begin{definition}\label{Defuentropy}
The \emph{conditional entropy of $f$ with respect to a measurable partition $\a$
given $\eta\in \P^u$} is defined as
$$h_\mu(f, \alpha|\eta)
=\limsup_{n\to \infty}\frac{1}{n}H_\mu(\alpha_0^{n-1}|\eta).
$$
The \emph{conditional entropy of $f$ given $\eta\in \P^u$}
is defined as
$$h_\mu(f|\eta)
=\sup_{\alpha \in \P}h_\mu(f, \alpha|\eta).
$$
and the \emph{unstable metric entropy of $f$} is defined as
\[
h_\mu^u(f)=\sup_{\eta\in \P^u}h_\mu(f|\eta).
\]
\end{definition}

%\begin{remark}
%In the definition of $h_\mu(f, \alpha|\eta)$ we take $\limsup$ instead of
%$\lim$.  This is because the sequence $\{H_\mu(\alpha_0^{n-1}|\eta)\}$ is not
%necessary subadditive, since $\eta$ is not invariant under $f$.
%Therefore, existence of such limit is not obvious.
%\end{remark}

The following theorem is one of the main results in \cite{HHW} (cf. Theorem A and Corollary A.2 therein).
\begin{theorem}\label{ThmA}(Cf. \cite{HHW})
For any $\a\in \P$ and $\eta\in \P^u$,
$$h_\mu^u(f)=h_\mu(f|\eta)=h_\mu(f, \alpha|\eta)=\lim_{n\to\infty}\frac{1}{n}H_\mu(\a_0^{n-1}|\eta).$$
\end{theorem}

\medskip
Unstable metric entropy is closely related to the entropy introduced by Ledrappier and Young (\cite{LY2}). Suppose that $f$ is $C^{1+\a} (\a>0)$ and $\mu$ is ergodic.
Recall a hierarchy of metric entropies
$h_\mu(f, \xi_i):=H_\mu(\xi_i|f\xi_i)$ introduced by Ledrappier and Young
in \cite{LY2},
where $i=1, \cdots, \tilde u$, and $\tilde u$
is the number of distinct positive Lyapunov exponents.
For each $i$, $\xi_i$ is an increasing partition subordinate to
the $i$th level of the unstable leaves $W^{(i)}$, and is a generator.
It is proved there that $h_{\mu}(f, \xi_{\tilde u})=h_\mu(f)$, the
metric entropy of $\mu$.
If there are $u (1\leq u\leq \tilde u)$ distinct Lyapunov exponents on unstable subbundle,
then the $u$th unstable foliation is exactly the unstable foliation of the
partially hyperbolic system $f$.
It is shown in \cite{HHW} that the unstable metric entropy $h_\mu^u(f)$ is identical
to $h_\mu(f, \xi_u)$ given by Ledrappier-Young. We remark that our definition of unstable metric entropy only requires $f$ to be $C^1$, while the definition and results by Ledrappier-Young requires the $C^{1+\a}$-regularity of $f$.

%Denote by $\Q^u$ the set of increasing partitions $\xi_u$
%that are subordinate to $W^{u}$, and are generators, that is, partitions $\xi_u$ satisfying condition~(i)-(iii)
%in Lemma.

Another notion introduced in \cite{HHW} is the unstable topological entropy $h^u_{\text{top}}(f)$.
As a generalization, we define the unstable topological pressure associated with a potential $\varphi \in C(M, \mathbb{R})$ as follows.
Denote by $d^u$ the metric induced by the Riemannian structure on the unstable manifold and let
$d^u_{n}(x,y)=\max _{0 \leq j \leq n-1}d^u(f^j(x),f^j(y))$.
Let $W^u(x,\delta)$ denote the open ball inside $W^u(x)$ with center $x$ and radius $\delta$ with respect to $d^u$. Let $E$ be a set of points in $\overline{W^u(x,\delta)}$ with pairwise $d^u_{n}$-distances at least $\epsilon$. We call $E$ an \emph{$(n,\epsilon)$ u-separated subset} of $\overline{W^u(x,\delta)}$. Put
\begin{equation*}
\begin{aligned}
P^u(f,\varphi, \epsilon,n,x,\delta):=\sup\{&\sum_{y\in E}\exp((S_n\varphi)(y))|\\
&E \text{\ is an\ } (n, \epsilon) \text{\ u-separated subset of\ }\overline{W^u(x,\delta)}\}
\end{aligned}
\end{equation*}
where $(S_n\varphi)(y)=\sum_{i=0}^{n-1}\varphi^i(y)$.
\begin{definition}\label{unstableentropy1}
We define \emph{unstable topological pressure} of $f$ with respect to \emph{the potential} $\varphi$ on $M$ to be
\begin{equation*}
\begin{aligned}
P^u(f, \varphi)&:=\lim_{\delta \to 0}\sup_{x\in M}P^u(f, \varphi, \overline{W^u(x,\delta)}),
\end{aligned}
\end{equation*}
where
\begin{equation*}
\begin{aligned}
P^u(f, \varphi, \overline{W^u(x,\delta)})&:=\lim_{\epsilon \to 0}\limsup_{n\to \infty}\frac{1}{n}\log P^u(f,\varphi, \epsilon,n,x,\delta).
\end{aligned}
\end{equation*}
\end{definition}

Two alternative ways to define unstable topological pressure are by using $(n,\epsilon)$ u-spanning sets and by using open covers. We discuss it in details in Section 2. Note that when $\varphi=0$, the unstable topological pressure reduces to the unstable topological entropy.

Let $\mathcal{M}_f(M)$ and $\mathcal{M}^e_f(M)$ denote
the set of all $f$-invariant and ergodic probability measures on $M$
respectively. Our first main result is the variational principle relating unstable topological pressure and
unstable metric pressure, the sum of unstable metric entropy and integral of the potential.

\begin{TheoremA}\label{TheoremA}
Let $f: M \to M$ be a $C^1$ partially hyperbolic diffeomorphism. Then for any $\varphi\in C(M, \mathbb{R})$,
\[P^u(f, \varphi)=\sup \Big\{h_\mu^u(f)+\int_M \varphi d\mu: \mu \in \mathcal{M}_f(M)\Big\}.\]
Moreover,
\[P^u(f, \varphi)=\sup \Big\{h_\mu^u(f)+\int_M \varphi d\mu: \mu \in \mathcal{M}^e_f(M)\Big\}.\]
\end{TheoremA}

As an immediate corollary, we recover the variational principle for unstable entropies, obtained in \cite{HHW}.
\begin{CorollaryA1}\label{AA1}
Let $f: M \to M$ be a $C^1$ partially hyperbolic diffeomorphism. Then
$$h^u_{\text{top}}(f)=\sup\{h_{\mu}^u(f): \mu \in \mathcal{M}_f(M)\}.
$$
Moreover,
$$
h^u_{\text{top}}(f)=\sup\{h_{\nu}^u(f): \nu \in \mathcal{M}^e_f(M)\}.
$$
\end{CorollaryA1}

Let $P(f, \varphi)$ be the classical topological pressure of $f$ associated to potential $\varphi$ (cf. Chapter 9 in \cite{W82}). By the definition of $P^u(f, \varphi)$ and Theorem A, we have the following facts.
\begin{CorollaryA2}\label{AA2}
$P^u(f, \varphi)\leq P(f, \varphi)$.

If $f$ is $C^{1+\alpha}$, the equation holds if there is no positive Lyapunov exponent in the
center direction at $\nu$-a.e. with respect to any ergodic measure $\nu$.
\end{CorollaryA2}

The variational principle Theorem A gives a natural way of selecting members of $\mathcal{M}_f(M)$.
The following concept of u-equilibrium state generalizes measure of maximal unstable entropy.

\begin{definition}\label{u-equilibrium}
Let $\varphi \in C(M, \mathbb{R})$. A member $\mu$ of $\mathcal{M}_f(M)$ is called a \emph{u-equilibrium state for $\varphi$} if $P^u(f, \varphi)=h_\mu^u(f)+\int \varphi d\mu$. Let $\mathcal{M}^u_\varphi(M,f)$ denote the set of all u-equilibrium states for $\varphi$.
\end{definition}

A measure of maximal unstable entropy is a u-equilibrium state for the potential $0$. A significant result in \cite{HHW} is that the unstable metric entropy function is upper semicontinuous (cf. Proposition 2.15 in \cite{HHW} which is restated in Lemma \ref{uppersemicontinuous} below). Therefore, a u-equilibrium state should always exist for partially hyperbolic diffeomorphisms. Furthermore, $\mathcal{M}^u_\varphi(M,f)$ has the following nontrivial properties.
\begin{TheoremB}
\begin{enumerate}
  \item $\mathcal{M}^u_\varphi(M,f)$ is convex.
  \item $\mathcal{M}^u_\varphi(M,f)$ is nonempty and compact.
  \item The extreme points of $\mathcal{M}^u_\varphi(M,f)$ are precisely
  the ergodic members of $\mathcal{M}^u_\varphi(M,f)$.
  \item If $\varphi, \psi \in C(M, \mathbb{R})$, and there exists $c\in \mathbb{R}$ such that $\varphi-\psi-c$ belongs to the closure of the set
  $\{h\circ f-h: h\in C(M, \mathbb{R})\}$ in $ C(M, \mathbb{R})$, then $\mathcal{M}^u_\varphi(M,f)=\mathcal{M}^u_\psi(M,f)$.
\end{enumerate}
\end{TheoremB}

\emph{Gibbs u-states} form a special class of invariant probability measures on $M$ whose conditional measures along unstable leaves are absolutely continuous with respect to the Lebesgue measure on the leaves (cf. \cite{PS}, \cite{BV}, see also p. 221 in \cite{BDV}).
In uniformly hyperbolic case, they correspond to Gibbs states, which are the equilibrium states associated to a potential defined by the Jacobian along unstable direction. If $f: M\to M$ is a $C^1$ partially hyperbolic diffeomorphism, there is also a distinguished potential $\varphi^u(x)=-\log \|Df|_{E^u(x)}\|$. We relate the u-equilibrium states associated to $\varphi^u$ to the Gibbs u-states of $f$.
\begin{TheoremC}
Let $f$ be $C^{1+\alpha}$ and $\mu \in \mathcal{M}_f(M)$. Then $\mu$ is a Gibbs u-state of $f$ if and only if $\mu$ is a u-equilibrium state of $\varphi^u$.
\end{TheoremC}

\begin{CorollaryC1}\label{CC1}
If $f$ is $C^{1+\alpha}$, then $P^u(f, \varphi^u)= 0$.
\end{CorollaryC1}

\begin{CorollaryC2}\label{CC2}
There always exists a Gibbs u-state for any $C^{1+\alpha}$ partially hyperbolic diffeomorphism.
\end{CorollaryC2}
Corollary C.2 which can be easily extended to the partially hyperbolic attractor case, recovers the existence result proved in \cite{PS}.

It is well known that the (classical) topological pressure $P(f, \cdot)$ determines the set $\mathcal{M}_f(M)$ and the entropy $h_\mu(f)$ for all $\mu \in \mathcal{M}_f(M)$, in the sense of Theorems 9.11 and 9.12 in \cite{W82}. We recall the precise meaning as follows. A \emph{finite signed measure} on $M$ is a map $\mu: \mathcal{B}(M) \to \mathbb{R}$ which is countably additive, where $\mathcal{B}(M)$ is the $\sigma$-algebra of Borel subsets of $M$. Then $\mu \in \mathcal{M}_f(M)$ if and only if $\int_M \varphi d\mu \leq P(f, \varphi)$ for $\forall \varphi \in C(M, \mathbb{R}).$ Moreover, $h_\nu(f)=\inf\left\{P(f, \varphi)-\int_M \varphi d\nu: \varphi \in C(M, \mathbb{R})\right\}$ holds if and only if the entropy map $\mu\mapsto h_\mu(f)$ is upper semi-continuous at $\nu$.
Rather surprisingly, the analogue holds for unstable pressure: the unstable topological pressure $P^u(f, \cdot)$ which might be considered as a partial pressure of the system, also determines the set $\mathcal{M}_f(M)$ and the entropy $h^u_\mu(f)$ for all $\mu \in \mathcal{M}_f(M)$. Moreover, we have a cleaner result since the unstable entropy map $\mu\mapsto h^u_\mu(f)$ is always upper semi-continuous.
\begin{TheoremD}
\begin{enumerate}
  \item Let $\mu: \mathcal{B}(M) \to \mathbb{R}$ be a finite signed measure. Then $\mu \in \mathcal{M}_f(M)$ if and only if $\int_M \varphi d\mu \leq P^u(f, \varphi)$, $\forall \varphi \in C(M, \mathbb{R}).$
  \item Let $\nu \in \mathcal{M}_f(M)$. Then
$$h_\nu^u(f)=\inf\left\{P^u(f, \varphi)-\int_M \varphi d\nu: \varphi \in C(M, \mathbb{R})\right\}.$$
\end{enumerate}

\end{TheoremD}

As the existence of u-equilibrium state is guaranteed by the upper semicontinuity of the unstable entropy map, it is natural to ask when the u-equilibrium state is unique. This question is very subtle and already attracts a lot of interest in the case of the classical pressure. In this paper, we study the differentiability properties of the unstable pressure and their relations to the uniqueness of u-equilibrium state. Such an approach is developed in \cite{W92} for the classical pressure.

To start with, we define a notion of tangent functional to the convex function $P^u(f, \cdot): C(M, \mathbb{R})\to \mathbb{R}$, which is closely related to the u-equilibrium state. The (classical) tangent functional can be found in Definition 9.9 in \cite{W82}.
\begin{definition}
Let $\varphi \in C(M, \mathbb{R})$. A \emph{u-tangent functional to $P^u(f, \cdot)$ at $\varphi$} is a finite signed measure $\mu: \mathcal{B}(M) \to \mathbb{R}$ such that
$$P^u(f, \varphi+\psi)-P^u(f, \varphi)\geq \int_M \psi d\mu,\quad \forall \psi\in C(M, \mathbb{R}).$$
Let $t_\varphi^u(M, f)$ denote the set of all u-tangent functionals to $P^u(f, \cdot)$ at $\varphi$.
\end{definition}

For the classical tangent functional and equilibrium state, the equality $\mathcal{M}_\varphi(M,f)=t_\varphi(M, f)$ holds under the assumption that $\mu\mapsto h_\mu(f)$ is upper semi-continuous at the members of $t_\varphi(M, f)$ (cf. Theorem 9.15 in \cite{W82}). The assumption is unnecessary for the u-tangent functional:
\begin{TheoremE}\label{equal}
$\mathcal{M}^u_\varphi(M,f)=t_\varphi^u(M, f).$
\end{TheoremE}

In the following, we consider two types of differentiability of unstable pressure.
\begin{definition}\label{Gateaux}
The unstable topological pressure $P^u(f, \cdot): C(M, \mathbb{R})\to \mathbb{R}$ is said to be \emph{Gateaux differentiable at $\varphi$} if
$$\lim_{t\to 0}\frac{1}{t}(P^u(f, \varphi+t\psi)-P^u(f, \varphi))$$
exists for any $\psi\in C(M, \mathbb{R})$.
\end{definition}

\begin{TheoremF}
$P^u(f, \cdot)$ is Gateaux differentiable at $\varphi$ if and only if there is a unique unstable tangent functional to $P^u(f, \cdot)$ at $\varphi$.
\end{TheoremF}
Combining Theorems E and F, we have
\begin{CorollaryF1}\label{CF1}
$P^u(f, \cdot)$ is Gateaux differentiable at $\varphi$ if and only if there is a unique u-equilibrium state of $\varphi$.
\end{CorollaryF1}

Now we consider the Fr\'{e}chet differentiability of unstable topological pressure.

\begin{definition}\label{frechet}
$P^u(f, \cdot): C(M, \mathbb{R})\to \mathbb{R}$ is said to be \emph{Fr\'{e}chet differentiable at $\varphi$} if $\exists \gamma \in C(M, \mathbb{R})^*$ such that
$$\lim_{\psi \to 0}\frac{|P^u(f, \varphi+\psi)-P^u(f, \varphi)-\gamma(\psi)|}{\|\psi\|}=0.$$
\end{definition}

Let $\mu_n \to \mu$ denote the convergence in weak$^*$ topology, and $\|\mu_n-\mu\|\to 0$ the convergence in norm topology on $\mathcal{M}_f(M)$. We have the following equivalent ways to describe Fr\'{e}chet differentiability of $P^u(f, \cdot)$.
\begin{TheoremG}
The following statements are mutually equivalent.\begin{enumerate}
\item $P^u(f, \cdot)$ is Fr\'{e}chet differentiable at $\varphi$.
\item  There is a measure $\mu_\varphi\in \mathcal{M}_f(M)$ such that whenever $(\mu_n)\subset \mathcal{M}_f(M)$ with $h_{\mu_n}^u(f)+\int_M \varphi d\mu_n\to P^u(f, \varphi)$ we have $\|\mu_n -\mu_\varphi\|\to 0$ as $n\to \infty.$
\item $t^u_\varphi(M,f)$ consists of one member $\mu_\varphi$ and
$$P^u(f, \varphi)>\sup\Big\{h_{\mu}^u(f)+\int_M \varphi d\mu: \mu \text{\ is ergodic and\ }\mu \neq \mu_\varphi\Big\}.$$
\item $t^u_\varphi(M,f)$ consists of one member $\mu_\varphi$ and there is a weak$^*$ neighborhood $V$ of $\mu_\varphi$ such that
$$h_{\mu_\varphi}^u(f)>\sup\{h_{\mu}^u(f): \mu\in V \text{\ is ergodic and\ }\mu \neq \mu_\varphi\}.$$
\item $P^u(f, \cdot)$ is affine on a neighborhood of $\varphi.$
\item $t^u_\varphi(M,f)$ consists of one member $\mu_\varphi$ and $\sup\{\|\mu-\mu_\varphi\|: \mu \in t^u_{\varphi+\psi}(M,f )\}\to 0$ as $\psi \to 0.$
\item $t^u_\varphi(M,f)$ consists of one member $\mu_\varphi$ and $\inf\{\|\mu-\mu_\varphi\|: \mu \in t^u_{\varphi+\psi}(M,f )\}\to 0$ as $\psi \to 0.$
\end{enumerate}
\end{TheoremG}
It follows that Fr\'{e}chet differentiability of $P^u(f, \cdot)$ implies the uniqueness of u-equilibrium state. It is also clear that Fr\'{e}chet differentiability of $P^u(f, \cdot)$ is stronger than Gateaux differentiability of $P^u(f, \cdot)$, either by the definitions or by Theorems F and G.

\section{Unstable topological pressure}\label{u-pressure}
In this section, we redefine the unstable topological pressure via the spanning sets and open covers, and discuss its basic properties.
\subsection{Definition using spanning sets}
Recall that unstable topological pressure is defined in
Definition~\ref{unstableentropy1} using $(n,\epsilon)$ u-separated sets.
We can also define unstable topological pressure by using
 $(n,\epsilon)$ u-spanning sets as follows.

%\subsection{Definition using separated sets and spanning sets}

 A set $F \subset W^u(x)$ is called an \emph{$(n,\epsilon)$ u-spanning set }of $\overline{W^u(x, \delta)}$ if $\overline{W^u(x, \delta)} \subset \bigcup_{y\in F}B^u_{n}(y,\epsilon)$, where $B^u_{n}(y,\epsilon)=\{z\in W^u(x): d^u_{n}(y,z) \leq\epsilon\}$ is the $(n,\epsilon)$ u-Bowen ball around $y$. Put
\begin{equation*}
\begin{aligned}
Q^u(f,\varphi, \epsilon,n,x,\delta):=\inf\Big\{&\sum_{x\in F}\exp((S_n\varphi)(x))|\\
&F \text{\ is an\ } (n, \epsilon) \text{\ u-spanning subset of\ }\overline{W^u(x,\delta)}\Big\}.
\end{aligned}
\end{equation*}
Then in Definition \ref{unstableentropy1} we can also define
$$P^u(f, \varphi, \overline{W^u(x,\delta)}):=\lim_{\epsilon \to 0}\limsup_{n\to \infty}\frac{1}{n}\log Q^u(f,\varphi, \epsilon,n,x,\delta).$$
It is standard to verify that these two definitions for $P^u(f, \varphi, \overline{W^u(x,\delta)})$ coincide.

The following lemma is useful.

\begin{lemma}\label{smalldelta}
$P^u(f, \varphi):=\sup_{x\in M}P^u(f, \varphi, \overline{W^u(x,\delta)})$ for any $\delta >0$.
\end{lemma}
\begin{proof}
It is easy to see that $P^u(f, \varphi)\leq \sup_{x\in M}P^u(f, \varphi, \overline{W^u(x,\delta)})$ for any $\delta >0$ since $\delta \mapsto \sup_{x\in M}P^u(f, \varphi, \overline{W^u(x,\delta)})$ is increasing.

Let us prove the other direction for some fixed $\delta >0$. For any $\rho >0$, there exists $y \in M$ such that
\begin{equation}\label{e:anyradius1}
\sup_{x\in M}P^u(f, \varphi, \overline{W^u(x,\delta)})\leq P^u(f, \varphi, \overline{W^u(y,\delta)})+\frac{\rho}{3}.
\end{equation}
Pick $\epsilon_0>0$ such that
\begin{equation}\label{e:anyradius2}
\begin{aligned}
P^u(f, \varphi, \overline{W^u(y,\delta)})&=\lim_{\epsilon \to 0}\limsup_{n\to \infty}\frac{1}{n}\log Q^u(f,\varphi, \epsilon,n,y,\delta) \\ &\leq \limsup_{n\to \infty}\frac{1}{n}\log Q^u(f,\varphi, \epsilon_0,n,y,\delta)+\frac{\rho}{3}.
\end{aligned}
\end{equation}
We can also choose $\delta_1>0$ small enough that such that $\delta_1<\delta$ and
\begin{equation}\label{e:anyradius3}
P^u(f, \varphi) \geq\sup_{x\in M}P^u(f, \varphi, \overline{W^u(x,\delta_1)})-\frac{\rho}{3}.
\end{equation}
Then there exist $y_i \in \overline{W^u(y,\delta)}, 1\leq i\leq N$ where $N$ only depends on $\delta$, $\delta_1$, and the Riemannian structure on $\overline{W^u(y,\delta)}$, such that
\begin{equation}\label{e:anyradius4}
\overline{W^u(y,\delta)} \subset \bigcup_{i=1}^N\overline{W^u(y_i, \delta_1)}.
\end{equation}
Then we have
\begin{equation*}
\begin{aligned}
\sup_{x\in M}P^u(f, \varphi, \overline{W^u(x,\delta)})&\leq P^u(f, \varphi, \overline{W^u(y,\delta)})+\frac{\rho}{3} \ \ \ \ \ \text{by\ } \eqref{e:anyradius1}\\
&\leq\limsup_{n\to \infty}\frac{1}{n}\log Q^u(f,\varphi, \epsilon_0,n,y,\delta)+\frac{2\rho}{3} \ \ \ \ \ \text{by\ } \eqref{e:anyradius2}\\
&\leq \limsup_{n\to \infty}\frac{1}{n}\log \left(\sum_{i=1}^N Q^u(f,\varphi, \epsilon_0,n,y_i,\delta_1)\right)+\frac{2\rho}{3} \ \ \ \ \ \text{by\ } \eqref{e:anyradius4}\\
&\leq \limsup_{n\to \infty}\frac{1}{n}\log N Q^u(f,\varphi, \epsilon_0,n,y_j,\delta_1)+\frac{2\rho}{3} \text{\ for some } 1\leq j \leq N \\
&=\limsup_{n\to \infty}\frac{1}{n}\log Q^u(f,\varphi, \epsilon_0,n,y_j,\delta_1)+\frac{2\rho}{3}\\
&\leq \lim_{\epsilon \to 0}\limsup_{n\to \infty}\frac{1}{n}\log Q^u(f,\varphi, \epsilon,n,y_j,\delta_1)+\frac{2\rho}{3}\\
& =P^u(f, \varphi, \overline{W^u(y_j,\delta_1)})+\frac{2\rho}{3}\\
&\leq \sup_{x\in M}P^u(f, \varphi, \overline{W^u(x,\delta_1)})+\frac{2\rho}{3}\\
& \leq P^u(f, \varphi)+\rho \ \ \ \ \ \text{by\ } \eqref{e:anyradius3}.
\end{aligned}
\end{equation*}
Since $\rho>0$ is arbitrary, we have $\sup_{x\in M}P^u(f, \varphi, \overline{W^u(x,\delta)})\leq P^u(f, \varphi)$.
\end{proof}

\subsection{Definition using open covers}
We proceed to define the unstable topological pressure by using open covers. Let $\mathcal{C}_M$ denote the set of Borel covers of $M$ and $\mathcal{C}_M^o\subset \mathcal{C}_M$ the set of open covers of $M$. Given $\mathcal{U}\in \mathcal{C}_M$, denote $\mathcal{U}_m^n:=\bigvee_{i=m}^n f^{-i}\mathcal{U}$. Put
$$p^u(f,\varphi, \mathcal{U}, n,x,\delta):=\inf\Big\{\sum_{B\in \mathcal{V}} \sup_{y\in B\cap \overline{W^u(x,\delta)}}\exp((S_n\varphi)(y))|\mathcal{V} \in \mathcal{C}_M, \mathcal{V}\succeq\ \mathcal{U}_0^{n-1}\Big\}.$$
If $B\cap \overline{W^u(x,\delta)}=\emptyset$, we set $\sup_{y\in B\cap \overline{W^u(x,\delta)}}\exp((S_n\varphi)(y))=0$.
\begin{definition}\label{unstableentropy2}
We define
\begin{equation*}
\begin{aligned}
\widetilde{P}^u(f, \varphi)&:=\lim_{\delta \to 0}\sup_{x\in M}\widetilde{P}^u(f, \varphi, \overline{W^u(x,\delta)}),
\end{aligned}
\end{equation*}
where
\begin{equation*}
\begin{aligned}
\widetilde{P}^u(f, \varphi, \overline{W^u(x,\delta)})&:=\sup_{\mathcal{U}\in \mathcal{C}_M^o}\limsup_{n\to \infty}\frac{1}{n}\log p^u(f,\varphi, \mathcal{U},n,x,\delta).
\end{aligned}
\end{equation*}
\end{definition}

\begin{remark}
It is not clear whether the sequence $\log p^u(f,\varphi, \epsilon,n,x,\delta)$ is subadditive or not, so we have used $\limsup$ in the definition above. This is one of the main difference from the case for classical topological pressure.
\end{remark}

Observe that for $\delta >0$ small enough, there exists $C>1$ such that for any $x\in M$,
\begin{equation}\label{e:distance}
d(y,z) \leq d^u(y,z)\leq Cd(y,z), \text{\ \ for any }y,z \in \overline{W^u(x,\delta)}
\end{equation}
since $M$ is compact and $W^u$ is a continuous foliation.
By arguments similar to the proof of Theorems 9.2 and 9.4 in \cite{W82}, we can verify that Definitions \ref{unstableentropy1} and \ref{unstableentropy2} for unstable topological pressure coincide:

\begin{proposition}
$\widetilde{P}^u(f, \varphi, \overline{W^u(x,\delta)})=P^u(f, \varphi, \overline{W^u(x,\delta)})$. As a consequence,
$$\widetilde{P}^u(f, \varphi)=P^u(f, \varphi).$$
\end{proposition}
\subsection{Basic properties of unstable topological pressure}
Here we list some properties of unstable topological pressure. The proof is straightforward by definition and hence is omitted.
\begin{proposition}\label{proppressure}
If $\varphi, \psi \in C(M, \mathbb{R})$ with norm $\|\cdot \|$ and $c\in \mathbb{R}$, then the following statements are true.
\begin{enumerate}
  \item $P^u(f, 0)=h_{\text{top}}^u(f)$.
  \item $P^u(f, \varphi+c)=P^u(f, \varphi)+c$.
  \item $\varphi \leq \psi$ implies $P^u(f, \varphi)\leq P^u(f, \psi)$. In particular, $h_{\text{top}}^u(f)+\inf \varphi \leq P^u(f, \varphi) \leq h_{\text{top}}^u(f)+\sup \varphi$.
  \item $|P^u(f, \varphi)-P^u(f, \psi)|\leq \|\varphi-\psi\|$.
  \item $P^u(f, \cdot)$ is convex.
  \item $P^u(f, \varphi+\psi\circ f-\psi)=P^u(f, \varphi)$.
  \item $P^u(f, \varphi+\psi)\leq P^u(f, \varphi)+P^u(f, \psi)$.
  \item $P^u(f, c\varphi)\leq cP^u(f, \varphi)$ if $c\geq 1$ and $P^u(f, c\varphi)\geq cP^u(f, \varphi)$ if $c\leq 1$.
  \item $|P^u(f, \varphi)|\leq P^u(f, |\varphi|)$.
\end{enumerate}
\end{proposition}

\section{The variational principle}\label{pfThmA}
\subsection{Some properties of unstable metric entropy}
In this subsection, we collect some important properties of unstable metric entropy proved in \cite{HHW}.
In particular, they will be used in the proof of variational principle (Theorem A) and in describing the set $\mathcal{M}^u_\varphi(M,f)$ in Theorem B.

\begin{lemma}\label{expansive}(Corollary A.2 in \cite{HHW})
$\disp h^u_\mu(f)=h_\mu(f, \a|\eta)
=\lim_{n\to\infty}\frac{1}{n}H_\mu(\a_0^{n-1}|\eta)$
for any $\a\in \P$ and $\eta\in \P^u$.
\end{lemma}

\begin{lemma}\label{Cpointwise}(Corollary 3.2 in \cite{HHW})
For any $\eta\in \P^u$ subordinate to unstable manifolds and any $\e>0$,
\begin{equation*}
h_\mu(f|\eta)%=H_\mu(f^{-1}\xi|\xi)
=\lim_{n\to \infty}-\frac{1}{n}\log\mu_x^{\eta}(B^u_{n}(x,\epsilon)) \quad
\mu \ae x.
\end{equation*}
\end{lemma}

\begin{lemma}\label{affine}(Proposition 2.14 in \cite{HHW})
For any $\a\in \P$ and $\eta\in \P^u$, the map
$\mu \mapsto H_\mu(\alpha|\eta)$
from $\mathcal{M}(M)$ to ${\mathbb R}^+\cup\{0\}$ is concave.
%{\color{blue}Why do we need it?  It is weaker than affineness.}

Furthermore, the map $\mu \mapsto h_\mu^u(f)$ from $\mathcal{M}_f(M)$ to
${\mathbb R}^+\cup\{0\}$ is affine.
\end{lemma}

Recall that for each partition $\a\in \P$, the partition $\zeta$
given by $\zeta(x)=\a(x)\cap W^u_\loc(x)$ for any $x\in M$
is denoted by $\a^u$.
Conversely, for each partition $\eta\in \P^u$, there is a partition
$\b\in \P$ such that $\eta(x)=\b(x)\cap W^u_\loc(x)$ for any $x\in M$.
Denote such $\b$ by $\eta^{\not u}$.

\begin{lemma}(Proposition 2.15 in \cite{HHW})\label{uppersemicontinuous}
(a) Let $\nu\in \mathcal{M}(M)$.
For any $\a\in \P$ and $\eta\in \P^u$ with $\mu(\partial \a)=0$
and $\mu(\partial \eta^{\not u})=0$,
the map $\mu \mapsto H_\mu(\alpha|\eta)$ from $\mathcal{M}(M)$ to
${\mathbb R}^+\cup\{0\}$ is upper semi-continuous at $\mu$,
i.e.
$$
\limsup_{\nu\to \mu}H_{\nu}(\alpha|\eta)\leq H_\mu(\alpha|\eta).
$$

(b) The unstable entropy map $\mu \mapsto h_\mu^u(f)$ from
$\mathcal{M}_f(M)$ to ${\mathbb R}^+\cup\{0\}$ is
upper semi-continuous at $\mu$.
i.e.
$$
\limsup_{\nu\to \mu}h^u_{\nu}(f)\leq h_\mu^u(f).
$$
\end{lemma}
The second part of the above lemma also follows from Theorem D of \cite{Yang}.

\subsection{Proof of the variational principle}
At first, we prove Proposition \ref{variationalprinciple1} stated below, which is an easier half of the variational principle Theorem A. The following lemma is well-known.
\begin{lemma}\label{convex}
Suppose $0\leq p_1, \cdots, p_m \leq 1, s=p_1+\cdots + p_m$ and $a_1, \cdots, a_m \in \mathbb{R}$. Then
$$\sum_{i=1}^mp_i(a_i-\log p_i)\leq s\left(\log \sum_{i=1}^m e^{a_i}-\log s\right).$$
\end{lemma}
The above lemma is almost identical to Lemma 1.24 in \cite{Bo1}, except that we have removed the condition $s \leq 1$.

\begin{proposition}\label{variationalprinciple1}

Let $\mu$ be any $f$-invariant probability measure. Then
$$h_\mu^u(f)+\int_M\varphi d\mu \leq P^u(f, \varphi).$$
\end{proposition}

\begin{proof}
Let $\mu=\int_{\mathcal{M}^e_f(M)}\nu d\tau(\nu)$ be the unique ergodic decomposition
where $\tau$ is a probability measure on the Borel subsets of $\mathcal{M}_f(M)$ and $\tau(\mathcal{M}^e_f(M))=1$.
Since $\mu \mapsto h_\mu^u(f)$ is affine and upper semi-continuous by Lemma \ref{affine} and \ref{uppersemicontinuous}, then so is $h_\mu^u(f)+\int_M \varphi d\mu$ and hence
\begin{equation}\label{e:ergodicdecom}
h_\mu^u(f)+\int_M \varphi d\mu=\int_{\mathcal{M}^e_f(M)}\Big(h_\nu^u(f)+\int_M \varphi d\nu\Big) d\tau(\nu)
\end{equation}
by a classical result in convex analysis (cf. Fact A.2.10 on p. 356 in \cite{Do}). So we only need to prove the proposition for
ergodic measures.

Suppose that $\mu$ is ergodic. Let $\rho>0$ be arbitrary. Take $\eta\in \P^u$ subordinate to unstable manifolds, and then take $\e>0$.
By Lemma~\ref{Cpointwise}, we have
\begin{equation*}
\lim_{n \to \infty}-\frac{1}{n}\log\mu_y^\eta(B_{n}^u(y,\e))
\geq h_\mu^u(f|\eta)  \qquad \mu-\text{a.e.}\ y.
\end{equation*}
Hence for $\mu$-a.e. $y$, there exists $N(y)=N(y,\e)>0$ such that if
$n\geq N(y)$, then
\begin{equation*}
\mu_y^\eta(B_{n}^u(y,\e)) \leq e^{-n(h_\mu^u(f|\eta)-\rho)},
\end{equation*}
and
\begin{equation}\label{e:function}
\frac{1}{n}(S_n\varphi)(y)\geq \int_M \varphi d\mu-\rho.
\end{equation}
Denote $E_n=E_n(\e)=\{y\in M: N(y)=N(y,\e)\leq n\}$.
Then $\mu\big(\cup_{n=1}^\infty E_n\big)=1$. So there exists $n>0$ large enough such that $\mu(E_n)>1-\rho$.
Hence, there exists $x\in M$ such that
$\mu_x^\eta(E_n)=\mu_x^\eta(E_n\cap \eta(x))>1-\rho$.
Fix such $n$ and $x$. If $y \in \eta(x)$, $\mu_y^\eta=\mu_x^\eta$.  We have
\begin{equation}\label{e:entropyestimate}
\mu_x^\eta(B_{n}^u(y,\e)) \leq e^{-n(h_\mu(f|\eta)-\rho)}, \qquad
\forall y\in E_n\cap \eta(x).
\end{equation}

Now we take $\delta>0$ such that $W^u(x, \delta)\supset \eta(x)$. Let $F$ be an $(n,\e/2)$ u-spanning set of $\overline{W^u(x, \delta)}\cap E_n$ satisfying
\begin{equation*}
\begin{aligned}
\overline{W^u(x,\delta)}\cap E_n \subset &\bigcup_{z\in F}B_{n}^u(z,\e/2),
\end{aligned}
\end{equation*}
and $B_{n}^u(z,\e/2)\cap E_n \neq \emptyset$ for any $z\in F$. Let $y(z)$ be an arbitrary point in $B_{n}^u(z,\e/2)\cap E_n$. We have
\begin{equation}\label{e:vpesti}
\begin{aligned}
1-\rho& < \mu_x^\eta(\overline{W^u(x,\delta)}\cap E_n)\leq \mu_x^\eta(\bigcup_{z\in F}B_{n}^u(z,\e/2))\\
&\leq \sum_{z\in F}\mu_x^\eta(B_{n}^u(z,\e/2))\leq \sum_{z\in F}\mu_x^\eta(B_{n}^u(y(z),\e)).
\end{aligned}
\end{equation}

Using \eqref{e:function}, \eqref{e:entropyestimate}, and then applying Lemma \ref{convex} with $p_i=\mu_x^\eta(B_{n}^u(y(z),\e))$, $a_i=(S_n\varphi)(y(z))$, we have
\begin{equation*}
\begin{aligned}
&\sum_{z\in F}\mu_x^\eta(B_{n}^u(y(z),\e))\left(n\left(\int_M \varphi d\mu-\rho\right)+n(h_\mu^u(f|\eta)-\rho)\right)\\
\leq &\sum_{z\in F}\mu_x^\eta(B_{n}^u(y(z),\e))((S_n\varphi)(y(z))-\log \mu_x^\eta(B_{n}^u(y(z), \e)))\\
\leq &\left(\sum_{z\in F}\mu_x^\eta(B_{n}^u(y(z),\e))\right)\left(\log \sum_{z\in F}\exp((S_n\varphi)(y(z)))-\log \sum_{z\in F}\mu_x^\eta(B_{n}^u(y(z),\e))\right).
\end{aligned}
\end{equation*}
Then combining \eqref{e:vpesti},
\begin{equation}
\begin{aligned}\label{e:estim}
&n\left(\int_M \varphi d\mu-\rho\right)+n(h_\mu^u(f|\eta)-\rho)\\
\leq &\log \sum_{z\in F}\exp((S_n\varphi)(y(z)))-\log \sum_{z\in F}\mu_x^\eta(B_{n}^u(y(z),\e))\\
\leq &\log \sum_{z\in F}\exp((S_n\varphi)(y(z)))-\log (1-\rho).
\end{aligned}
\end{equation}
Let $\tau_\e:=\{|\varphi(x)-\varphi(y)|:d(x,y)\leq \e\}$. Then for any $z\in F$, $\exp((S_n\varphi)(y(z)))\leq \exp((S_n\varphi)(z)+n\tau_\e)$. Dividing by $n$ and taking the $\limsup$ on both sides of \eqref{e:estim}, we have
\begin{equation*}
\begin{aligned}
\int_M \varphi d\mu+h_\mu^u(f|\eta)-2\rho \leq \limsup_{n\to \infty}\frac{1}{n}\log \sum_{z\in F}\exp((S_n\varphi)(z))+\tau_\e.
\end{aligned}
\end{equation*}
Moreover, we can choose a sequence of $F$ such that
$$\limsup_{n\to \infty}\frac{1}{n}\log \sum_{z\in F}\exp((S_n\varphi)(z))\leq P^u(f, \varphi).$$
Since $\rho>0$ is arbitrary and $\tau_\e \to 0$ as $\e \to 0$, one has $\int_M \varphi d\mu+h_\mu^u(f|\eta) \leq P^u(f, \varphi)$.

%Let ${\mathcal B}_0$ be the sub-$\sigma$-algebra generated by invariant sets
%of $f$ and $\zeta$ be the partition determined by ${\mathcal B}_0$.
%Let $\{\mu^\zeta_x\}$ be the family of conditional measures with respect
%to $\zeta$.
%Hence for $\mu$-a.e. $x\in M$, $\mu^\zeta_x$ is an ergodic measure.

%Take any $\a\in\P$, $\eta\in \P^u$.
%Since for $\mu$-almost every $x\in M$, the entire unstable manifold $W^u(x)$
%is contained in $\zeta(x)$, we have $\eta(x)\subset \zeta(x)$.
%Consequently, $(\mu_x^\zeta)_x^\eta=\mu_x^\eta$.

%Thus
%\begin{equation}\label{fdecomp}
%\begin{aligned}
%h_\mu^u(f)+\int_M \varphi d\mu=&h_\mu(f, \a|\eta)+\int_M \varphi d\mu\\
%=&\int_M \lim_{n\to\infty}\frac{1}{n}I_\mu(\a_0^{n-1}|\eta)(x) d\mu +\int_M \varphi d\mu\\
%=&\int_M \left(h_{\mu^\zeta_x}(f, \a|\eta)+\int_M \varphi d\mu_x^\zeta\right)d\mu(x)\\
%\leq& \int_M P^u(f, \varphi)d\mu(x)=P^u(f, \varphi).
%\end{aligned}
%\end{equation}
\end{proof}

\begin{proof}[Proof of Theorem A]
 We start to prove that for any $\rho>0$, there exists $\mu \in \mathcal{M}_f(M)$ such that $h_{\mu}^u(f)+\int_M \varphi d\mu \geq P^u(f, \varphi)-\rho$. Combining with Proposition \ref{variationalprinciple1}, we obtain the first equality in Theorem A.

%Take $\e>0$ small enough.
%Let $S_n$ be an $(n,\e)$ u-separated set of $\overline{W^u(x,\delta)}$
%with cardinality $N^u(f,\e,n,x,\delta)$.
For some $\delta>0$ small enough, we can find a point $x\in M$ such that
$$P^u(f, \varphi, \overline{W^u(x,\delta)})\geq P^u(f, \varphi)-\rho.$$
Take $\e>0$ small enough. Let $E_n$ be an $(n,\e)$ u-separated set of $\overline{W^u(x,\delta)}$ with cardinality $N^u(f,\e,n,x,\delta)$ such that
$$\log \sum_{y\in E_n}\exp ((S_n\varphi)(y)) \geq \log P^u(f,\varphi, \e,n,x,\delta)-1. $$
Define
$$\nu_n:=\frac{\sum_{y\in E_n}\exp((S_n\varphi)(y))\delta_y}{\sum_{z\in E_n}\exp((S_n\varphi)(z))},$$
and
$$\mu_n:=\frac{1}{n}\sum_{i=0}^{n-1}f^i\nu_n.$$
Since the set $\mathcal{M}(M)$ of all probability measures on $M$ is a compact space with weak* topology, there exists a subsequence $\{n_k\}$ of natural numbers such that $\lim_{k\to \infty}\mu_{n_k}=\mu$. Obviously $\mu \in \mathcal{M}_f(M)$.

We can choose a partition $\eta\in\P^u$ such that $W^u(x,\delta)\subset \eta(x)$ (by shrinking $\delta$ if necessary).
That is, $W^u(x,\delta)$ is contained in a single element of $\eta$.
Then choose $\a\in \P$ such that $\mu(\partial\alpha)=0$,
and $\text{diam}(\alpha) < \frac{\e}{C}$
where $C>1$ is as in \eqref{e:distance}.
Hence we have $\log N^u(f,\e,n,x,\delta)=H_{\nu_n}(\alpha_0^{n-1}|\eta)$.

Fix a natural numbers $q>1$.  For any natural number $n>q$,
$j=0, 1, \cdots, q-1$, put $a(j)=[\frac{n-j}{q}]$, where
$[a]$ denotes the integer part of $a>0$.
Then
$$\bigvee_{i=0}^{n-1}f^{-i}\alpha
=\bigvee_{r=0}^{a(j)-1}f^{-(rq+j)}\alpha_0^{q-1} \vee \bigvee_{t\in S_j}f^{-t}\alpha,$$
where $S_j=\{0,1,\cdots, j-1\}\cup \{j+qa(j), \cdots, n-1\}$.

For a partition $\a\in \P$, denote by $\a^u$ the partition in $\P^u$
whose elements are given by $\a^u(x)=\a(x)\cap W^u_\loc(x)$. Note that
\[
f^{rq}\big(\bigvee_{i=0}^{r-1}f^{-iq}\alpha_0^{q-1}\vee f^j\eta\big)
=f^{rq}\big(\alpha_0^{rq-1}\vee f^j\eta\big)
=f\a\vee \cdots \vee f^{rq}\a\vee f^{rq+j}\eta
\ge f\a^u.
\]

We can get that
\begin{equation}\label{fVP1}
\begin{aligned}
&H_{\nu}(\bigvee_{r=0}^{a(j)-1}f^{-rq}\alpha_0^{q-1}|f^j\eta)\\
= &H_{\nu}(\alpha_0^{q-1}|f^j\eta)
+ \sum_{r=1}^{a(j)-1}H_{f^{rq}\nu}\Big(\alpha_0^{q-1}
 \Big|f^{rq}\big(\bigvee_{i=0}^{r-1}f^{-iq}\alpha_0^{q-1}\vee f^j\eta\big) \Big)\\
\le&  H_{\nu}(\alpha_0^{q-1}|f^j\eta)
+ \sum_{r=1}^{a(j)-1}H_{f^{rq}\nu}(\alpha_0^{q-1}|f\a^u).
\end{aligned}
\end{equation}
Also,
\begin{equation}\label{fVP2}
H_{\nu}(\bigvee_{r=0}^{a(j)-1}f^{-(rq+j)}\alpha_0^{q-1}|\eta)
=H_{f^j\nu}(\bigvee_{r=0}^{a(j)-1}f^{-rq}\alpha_0^{q-1}|f^j\eta).
\end{equation}
Replacing $\nu$ by $\nu_n$ and $f^j\nu_n$ in \eqref{fVP2} and  \eqref{fVP1}
respectively we get
\begin{equation*}
\begin{aligned}
&\ \ \ \log \sum_{y\in E_n}\exp ((S_n\varphi)(y))\\
&=\sum_{y\in E_n} \nu_n(\{y\})\big(-\log \nu_n(\{y\})+(S_n\varphi)(y)\big)\\ %\\
&=H_{\nu_n}(\alpha_0^{n-1}|\eta)+\int_M (S_n\varphi) d\nu_n\\
&=H_{\nu_n}\big(\bigvee_{r=0}^{a(j)-1}f^{-(rq+j)}\alpha_0^{q-1} \vee
   \bigvee_{t\in S_j}f^{-t}\alpha|\eta\big)+\int_M (S_n\varphi) d\nu_n\\
&\leq \sum_{t\in S_j}H_{\nu_n}(f^{-t}\alpha|\eta)
 +H_{\nu_n}\big(\bigvee_{r=0}^{a(j)-1}f^{-rq-j}\alpha_0^{q-1}|\eta\big)+\int_M (S_n\varphi) d\nu_n\\
&\leq \sum_{t\in S_j}H_{\nu_n}(f^{-t}\alpha|\eta)
 +H_{f^j\nu_n}\big(\bigvee_{r=0}^{a(j)-1}f^{-rq}\alpha_0^{q-1}|f^j\eta\big)+\int_M (S_n\varphi) d\nu_n\\
&\leq \sum_{t\in S_j}H_{\nu_n}(f^{-t}\alpha|\eta)
+ H_{f^j\nu_n}\big(\alpha_0^{q-1}|f^j\eta\big)
+ \sum_{r=1}^{a(j)-1}H_{f^{rq+j}\nu_n}(\alpha_0^{q-1}|f\a^u)+\int_M (S_n\varphi) d\nu_n.
\end{aligned}
\end{equation*}
It is clear that $\text{card} S _j \leq 2q$.
Denote by $d$ the number of elements of $\a$.
Summing the inequalities over $j$ form $0$ to $q-1$ and
dividing by $n$, by Lemma \ref{affine} we get
\begin{equation}\label{e:mu}
\begin{aligned}
&\frac{q}{n}\log \sum_{y\in E_n}\exp ((S_n\varphi)(y)) \\
\leq &\frac{1}{n}\sum_{j=0}^{q-1} \sum_{t\in S_j}H_{\nu_n}(f^{-t}\alpha|\eta)
   +\frac{1}{n}\sum_{j=0}^{q-1}H_{f^j\nu_n}(\alpha_0^{q-1}|f^j\eta)
   +\frac{1}{n}\sum_{i=0}^{n-1}H_{f^i\nu_n}(\alpha_0^{q-1}|f\a^u)+\frac{q}{n}\int_M (S_n\varphi) d\nu_n\\
\leq &\frac{2q^2}{n}\log d
+\frac{1}{n}\sum_{j=0}^{q-1}H_{f^j\nu_n}(\alpha_0^{q-1}|f^j\eta)
 +H_{\mu_n}(\alpha_0^{q-1}|f\a^u)+q\int_M \varphi d\mu_n.
\end{aligned}
\end{equation}
Let $\{n_k\}$ be a sequence of natural numbers such that
\begin{enumerate}
  \item $\mu_{n_k} \to \mu$ as $k\to \infty$;
  \item $\disp \lim_{k\to \infty}\frac{1}{n_k}\log P^u(f,\varphi, \e,n_k,x,\delta)= \limsup_{n\to \infty}\frac{1}{n}\log P^u(f,\varphi, \e,n,x,\delta)$.
%\item $\nu_{n_k} \to \nu$ as $k\to \infty$ for some $\nu\in \mathcal{M}(M)$.
\end{enumerate}
Since $\mu(\partial \alpha)=0$, and $\mu$ is invariant,
$\mu(\partial \alpha_0^{q-1})=0$ for any $q\in {\mathbb N}$. By Lemma \ref{uppersemicontinuous},
$$
\limsup_{k \to \infty}H_{\mu_{n_k}}(\alpha_0^{q-1}|f\a^u)
\leq H_\mu(\alpha_0^{q-1}|f\a^u).
$$
Thus replacing $n$ by $n_k$ in \eqref{e:mu} and letting $k\to \infty$,
we get
$$q\limsup_{n\to \infty}\frac{1}{n}\log P^u(f,\varphi, \e,n,x,\delta) \leq H_{\mu}(\alpha_0^{q-1}|f\a^u)+q\int_M \varphi d\mu.$$
Then by Lemma \ref{expansive},
\begin{equation*}
P^u(f, \varphi, \overline{W^u(x,\delta)})
\leq \lim_{q \to \infty}\frac{1}{q}H_{\mu}(\alpha_0^{q-1}|f\a^u)+\int_M \varphi d\mu
=h_\mu^u(f)+\int_M \varphi d\mu.
\end{equation*}
Thus $h_\mu^u(f)+\int_M \varphi d\mu\geq P^u(f, \varphi)-\rho$.
Since $\rho$ is arbitrary, we get by combining Proposition \ref{variationalprinciple1}
\begin{equation*}
P^u(f, \varphi)=\sup \Big\{h_\mu^u(f)+\int_M \varphi d\mu: \mu \in \mathcal{M}_f(M)\Big\}.
\end{equation*}

We prove the second equation in Theorem A.

Let $\rho>0$ be sufficiently small.
Then there exists an invariant measure $\mu$ such that
$h_\mu^u(f)+\int_M \varphi d\mu > P^u(f, \varphi)-\rho/2$. By \eqref{e:ergodicdecom}, there exists an ergodic measure $\nu$ such that
$$h_{\nu}^u(f)+\int_M \varphi d\nu>h_\mu^u(f)+\int_M \varphi d\mu-\rho/2> P^u(f, \varphi)-\rho.$$
Since $\rho$ is arbitrary, we have
\[P^u(f, \varphi)=\sup \Big\{h_\mu^u(f)+\int_M \varphi d\mu: \mu \in \mathcal{M}^e_f(M)\Big\}.\]
\end{proof}

The proof of Corollary A.1 is straightforward, hence is omitted.
\begin{proof}[Proof of Corollary A.2]
The inequality $P^u(f, \varphi)\leq P(f, \varphi)$ follows from the definition directly.

If $f$ is $C^{1+\alpha}$ and there is no positive Lyapunov exponents in the center direction, then by Ledrappier-Young formula \cite{LY2} and Theorem A in \cite{HHW}, $h_\mu(f)= h_\mu^u(f)$ for any $\mu\in \mathcal{M}^e_f(M)$. Then by Theorem A and the classical variational principle for pressure (cf. Theorem 9.10 in \cite{W82}),
\begin{equation*}
\begin{aligned}
P^u(f, \varphi)&=\sup \Big\{h_\mu^u(f)+\int_M \varphi d\mu: \mu \in \mathcal{M}^e_f(M)\Big\}\\
&=\sup \Big\{h_\mu(f)+\int_M \varphi d\mu: \mu \in \mathcal{M}^e_f(M)\Big\}\\
&=P(f, \varphi).
\end{aligned}
\end{equation*}
\end{proof}

\section{U-equilibrium states}\label{uequi}
In this section, we shall first give some fundamental properties for the set of u-equilibrium states, proving Theorem B. Then for the  particular potential $\varphi^u=-\log \|Df|_{E^u}\|$ we relate the u-equilibrium states at $\varphi^u$ to the Gibbs u-states of $f$.
%\subsection{Unstable entropy map}

%{\color{red}The next two lemmas are proved in [Hu-Hua-Wu], so the proof should be omitted.  Probably we should state them as lemmas,}
%\begin{proposition}(Cf. Proposition 3.1 in \cite{LS})\label{partition1}
%Let $\mu\in \mathcal{M}_f(M)$. Then there exists a measurable partition $\xi$ of $M$ such that
%\begin{enumerate}
  %\item $\xi$ is a partition subordinate to $W^u$, i.e. for $\mu$-almost every $x$, $\xi(x) \subset W^u(x)$ and contains an open neighborhood of $x$ in $W^u(x)$;
  %\item $\xi$ is increasing, i.e. $f^{-1}\xi \geq \xi$;
  %\item $\bigvee_{n=1}^\infty f^{-n}\xi=\varepsilon$, where $\varepsilon$ is the partition of $M$ into points;
  %\item $\bigwedge_{n=1}^\infty f^n\xi=\mathcal{H}(\Pi^+)$, where $\Pi^+$ is the partition of $M$ into global unstable manifolds, and $\mathcal{H}(\Pi^+)$ is the measurable hull of $\Pi^+$.
%\end{enumerate}
%\end{proposition}
%In the remainder of the paper, we always let $\xi$ be the measurable partition as in Proposition \ref{partition1}. The following lemmata will be used in the proof of the variational principle Theorem A.
%\begin{lemma}\label{countablepartition1}(Cf. Corollary 2.12 in \cite{HHW})
%$h_\mu(f,\xi)=\lim_{n\to \infty}\frac{1}{n}H((\alpha\vee \xi)_0^{n-1}|\xi)$ for any $\alpha \in \mathcal{P}_M$.
%\end{lemma}

\subsection{Properties of $\mathcal{M}^u_\varphi(M,f)$}

\begin{proof}[Proof of Theorem B]
It follows from Lemma \ref{affine} that $\mu \mapsto h_\mu^u(f)+\int_M \varphi d\mu$ is affine. Hence (1) holds.

It follows from Lemma \ref{uppersemicontinuous} that $\mu \mapsto h_\mu^u(f)+\int_M \varphi d\mu$ is upper semi-continuous.
The set $\mathcal{M}^u_\varphi(M,f)$ is nonempty because an upper semi-continuous function on a compact space attains its supremum.
If $\mu_n \in \mathcal{M}^u_\varphi(M,f)$ and $\mu_n \to \mu$ in $\mathcal{M}_f(M)$,
then $h_\mu^u(f)+\int_M \varphi d\mu\geq \limsup_{n\to \infty}h_{\mu_n}^u(f)+\int_M \varphi d\mu_n=P^u(f, \varphi)$. This together with Theorem A proves that $\mathcal{M}^u_\varphi(M,f)$ is compact.

If $\mu \in \mathcal{M}^u_\varphi(M,f)$ is ergodic, then it is an extreme point of $\mathcal{M}_f(M)$ and hence of $\mathcal{M}^u_\varphi(M,f)$.
Now let $\mu \in \mathcal{M}^u_\varphi(M,f)$ be an extreme point of $\mathcal{M}^u_\varphi(M,f)$, and suppose that $\mu=p\mu_1+(1-p)\mu_2$ for some $p\in [0,1]$. Combining Lemma \ref{affine}, $P^u(f, \varphi)=h_\mu^u(f)+\int_M \varphi d\mu=p(h_{\mu_1}^u(f)+\int_M \varphi d\mu_1)+(1-p)(h_{\mu_2}^u(f)+\int_M \varphi d\mu_2)$. By the variational principle Theorem A, we must have $\mu_1, \mu_2 \in \mathcal{M}^u_\varphi(M,f)$. Hence $\mu_1=\mu_2=\mu$ and $\mu$ is an extreme point of $\mathcal{M}_f(M)$. Thus $\mu$ is ergodic. This proves (3).

Now we prove (4). By Proposition \ref{proppressure} (2), (4) and (6), $P^u(f, \varphi)=P^u(f, \psi)+c$. On the other hand, it is easy to see that $\int_M \varphi d\mu=\int_M \psi d\mu+c$, and hence $h_\mu^u(f)+\int_M \varphi d\mu=h_\mu^u(f)+\int_M \psi d\mu+c$. Thus $\mathcal{M}^u_\varphi(M,f)=\mathcal{M}^u_\psi(M,f)$.
\end{proof}

\subsection{Gibbs u-states}
There are two leading cases for a potential $\varphi$. First, $\varphi$ is the constant function $0$. In this case, the unstable topological pressure is just the unstable topological entropy (Proposition \ref{proppressure}(1)) and Theorem B(2) gives existence of measure of maximal unstable metric entropy.

Second,  $\varphi^u(x):=-\log \|\det Df|_{E^u}(x)\|$. Theorem B(2) gives existence of u-equilibrium states with respect to $\varphi^u$. We start to prove Theorem C, which claims that when $f$ is $C^{1+\alpha}$ such u-equilibrium states coincide with Gibbs u-states first studied in \cite{PS}.

\begin{lemma}\label{gibbs}(Cf. Proposition 5.2 in \cite{Yang})
If $f$ is $C^{1+\alpha}$ and $\mu \in \mathcal{M}_f(M)$, then
$$h^u_\mu(f)\leq \int_M -\varphi^u d\mu.$$
The equality holds if and only if $\mu$ is a Gibbs u-state of $f$.
\end{lemma}
We characterize Gibbs u-states of $f$ by u-equilibrium states of $f$ with respect to $\varphi^u$:

\begin{proof}[Proof of Theorem C]
By Theorem A and Lemma \ref{gibbs},
$$P^u(f, \varphi^u)=\sup\Big\{h_\mu^u(f)+\int_M \varphi^u d\mu\Big\} = 0.$$
By Lemma \ref{gibbs} again, $\mu$ is a Gibbs u-state of $f$ if and only if $\mu$ is a u-equilibrium state of $\varphi^u$.
\end{proof}

Corollary C.1 is already obtained in the proof of Theorem C.
%{\color{red} I think that the meterials in this subsection should go either Section 3, or 5, depending which one is more closely related.}
\begin{proof}[Proof of Corollary C.2]
Since u-equilibrium state for any continuous function $\varphi$ always exists by Theorem B(2), we know from Theorem C that Gibbs u-state always exists.
\end{proof}

\section{Unstable topological pressure determines $\mathcal{M}_f(M)$} \label{determin}

\begin{proof}[Proof of Theorem D]
(1) If $\mu \in \mathcal{M}_f(M)$, then $\int_M \varphi d\mu \leq P^u(f, \varphi)$ by the variational principle Theorem A. Now let  $\mu: \mathcal{B}(M) \to \mathbb{R}$ be a finite signed measure such that $\int_M \varphi d\mu \leq P^u(f, \varphi)$, $\forall \varphi \in C(M, \mathbb{R}).$ First we show that $\mu$ is a measure. Let $\varphi\geq 0$. If $\epsilon>0$ and $n>0$ is large enough, then
\begin{equation*}
\begin{aligned}
\int_M n(\varphi+\epsilon)d\mu&=-\int_M -n(\varphi+\epsilon)d\mu\\
&\geq -P^u(f, -n(\varphi+\epsilon))\\
&\geq -(h_{\text{top}}^u(f)+\sup(-n(\varphi+\epsilon)))\\
&=-h_{\text{top}}^u(f)+n\inf(\varphi+\epsilon)>0.
\end{aligned}
\end{equation*}
Hence $\int_M (\varphi+\epsilon)d\mu>0$. Thus $\int_M \varphi d\mu\geq 0$ and $\mu$ is a measure.

Next we show $\mu$ is a probability measure. For $n\in \mathbb{Z}$, $\int_M nd\mu\leq P^u(f,n)=h_{\text{top}}^u(f)+n.$ If $n>0$, then $\mu(M)\leq \frac{1}{n}h_{\text{top}}^u(f)+1$. Hence $\mu(M)\leq 1.$ If $n<0$, then $\mu(M)\geq \frac{1}{n}h_{\text{top}}^u(f)+1$. Hence $\mu(M)\geq 1.$ It follows that $\mu(M)=1.$

At last we show $\mu \in \mathcal{M}_f(M)$. For $n\in \mathbb{Z}$, $n\int_M (\varphi\circ f-\varphi)d\mu \leq P^u(f, n(\varphi\circ f-\varphi))=h_{\text{top}}^u(f)$ by Proposition \ref{proppressure} (6) and (1). If $n>0$, then $\int_M (\varphi\circ f-\varphi)d\mu\leq \frac{1}{n}h_{\text{top}}^u(f)$. Hence $\int_M (\varphi\circ f-\varphi)d\mu\leq 0.$ If $n<0$, then $\int_M (\varphi\circ f-\varphi)d\mu\geq \frac{1}{n}h_{\text{top}}^u(f)$. Hence $\int_M (\varphi\circ f-\varphi)d\mu\geq 0.$ Therefore $\int_M \varphi\circ f d\mu=\int_M \varphi d\mu.$ So $\mu \in \mathcal{M}_f(M)$.

(2)
By the variational principle, $h_\nu^u(f)\leq \inf\left\{P^u(f, \varphi)-\int_M \varphi d\nu: \varphi \in C(M, \mathbb{R})\right\}.$ To prove the other direction, let $b>h_\nu^u(f)$. Put
$$C=\{(\mu, t)\in \mathcal{M}_f(M)\times \mathbb{R}: 0\leq t\leq h_\mu^u(f)\}.$$
By Lemma \ref{affine}, $C$ is a convex subset of $C(M, \mathbb{R})^* \times \mathbb{R}$, where the weak$^*$-topology is used on $C(M, \mathbb{R})^*$. Then $(\nu, b)\notin \bar{C}$ as $\mu \mapsto h_\mu^u(f)$ is upper semi-continuous at $\nu$. By the classical result, there is a continuous linear functional $F: C(M, \mathbb{R})^*\times \mathbb{R} \to \mathbb{R}$ such that $F(\mu, t)\leq F(\nu, b)$ for any $(\mu, t)\in \bar{C}$. Suppose that $F$ has the form $F(\mu,t)=\int_M \psi d\mu+td$ for some $\psi \in C(M, \mathbb{R})$ and $d\in \mathbb{R}$. Then $\int_M \psi d\mu+td<\int_M \psi d\nu+bd$ for any $(\mu, t)\in \bar{C}$. In particular, $\int_M \psi d\mu+h_\mu^u(f)d<\int_M \psi d\nu+bd$ for any $\mu\in \mathcal{M}_f(M)$. Setting $\mu=\nu$, we have $h_\nu^u(f)d<bd$. Hence $d>0$. We have
$$\int_M \frac{\psi}{d} d\mu+h_\mu^u(f)<\int_M \frac{\psi}{d} d\nu+b$$
for any $\mu \in \mathcal{M}_f(M)$. By the variational principle,
$$P^u\Big(f, \frac{\psi}{d}\Big)\leq \int_M \frac{\psi}{d} d\nu+b.$$
Then
$$b\geq P^u\Big(f, \frac{\psi}{d}\Big)- \int_M \frac{\psi}{d} d\nu\geq \inf\{P^u(f, \varphi)-\int_M\varphi d\nu: \varphi \in C(M, \mathbb{R})\}.$$
Therefore, $h_\nu^u(f)\geq \inf\{P^u(f, \varphi)-\int_M\varphi d\nu: \varphi \in C(M, \mathbb{R})\}.$
\end{proof}

\section{Differentiability properties of the unstable topological pressure}\label{Differentiability}
In this section, we consider the differentiability properties of the unstable topological pressure. We shall first give the relation between the u-tangent functionals and the u-equilibrium states, then consider the Gateaux differentiability and Fr\'{e}chet differentiability of the unstable topological pressure. The equivalence of the Gateaux differentiability of  $P^u(f, \cdot)$ and the existence of unique unstable tangent functional $P^u(f, \cdot)$ at a given $\varphi$ is obtained, and several necessary and
sufficient conditions for $P^u(f, \cdot)$ to be Fr\'{e}chet differentiable at a given $\varphi$ are given.

\subsection{U-tangent functionals}
\begin{proof}[Proof of Theorem E]
If $\mu \in \mathcal{M}^u_\varphi(M,f)$, then $P^u(f, \varphi)=h_\mu^u(f)+\int_M \varphi d\mu$. We have
\begin{equation*}
\begin{aligned}
P^u(f, \varphi+\psi)-P^u(f, \varphi)&\geq h_\mu^u(f)+\int_M (\varphi+\psi)d\mu-h_\mu^u(f)-\int_M \varphi d\mu\\
&=\int_M \psi d\mu \quad \forall \psi \in C(M, \mathbb{R})
\end{aligned}
\end{equation*}
where the variational principle Theorem A is used in the first inequality.
Therefore $\mu \in t_\varphi^u(M, f).$

Conversely, let $\mu \in t_\varphi^u(M, f).$ Then for $\forall \psi \in C(M, \mathbb{R})$,
\begin{equation*}
\begin{aligned}
P^u(f, \varphi+\psi)-P^u(f, \varphi)\geq \int_M \psi d\mu=\int_M (\varphi+\psi)d\mu-\int_M \varphi d\mu,
\end{aligned}
\end{equation*}
which implies that
\begin{equation*}
\begin{aligned}
P^u(f, \varphi+\psi)-\int_M (\varphi+\psi)d\mu\geq P^u(f, \varphi)-\int_M \varphi d\mu.
\end{aligned}
\end{equation*}
Since $\psi \in C(M, \mathbb{R})$ is arbitrary, one has
\begin{equation*}
\begin{aligned}
\inf \Big\{P^u(f, h)-\int_M hd\mu: h\in C(M, \mathbb{R})\Big\}\geq P^u(f, \varphi)-\int_M \varphi d\mu.
\end{aligned}
\end{equation*}
By Theorem D(2), we have $h_\mu^u(f)\geq P^u(f, \varphi)-\int_M \varphi d\mu$. Combining with the variational principle Theorem A, we have $P^u(f, \varphi)=h_\mu^u(f)+\int_M \varphi d\mu$. Thus $\mu \in \mathcal{M}^u_\varphi(M,f).$
\end{proof}

\subsection{Gateaux differentiability}
Since $P^u(f, \cdot)$ is convex (Proposition \ref{proppressure}(5)), for any $\varphi, \psi\in C(M, \mathbb{R})$ the map $t\mapsto \frac{1}{t}(P^u(f, \varphi+t\psi)-P^u(f, \varphi))$ is increasing and hence the following two limits exist:
\begin{definition}
$$d^+P^u(f, \varphi)(\psi):=\lim_{t\to 0^+}\frac{1}{t}(P^u(f, \varphi+t\psi)-P^u(f, \varphi))$$
and
$$d^-P^u(f, \varphi)(\psi):=\lim_{t\to 0^-}\frac{1}{t}(P^u(f, \varphi+t\psi)-P^u(f, \varphi)).$$
\end{definition}
The following proposition is immediate.
\begin{proposition}\label{derivative}
\begin{enumerate}
  \item $d^-P^u(f, \varphi)(\psi)=-d^+P^u(f, \varphi)(-\psi)$,
  \item $d^-P^u(f, \varphi)(\psi)\leq d^+P^u(f, \varphi)(\psi)$.
\end{enumerate}
\end{proposition}
Recall that the unstable topological pressure $P^u(f, \cdot)$ is said to be Gateaux differentiable at $\varphi$ if
$$\lim_{t\to 0}\frac{1}{t}(P^u(f, \varphi+t\psi)-P^u(f, \varphi))$$
exists for any $\psi\in C(M, \mathbb{R})$ (See Definition \ref{Gateaux}). By Proposition \ref{derivative}, $P^u(f, \cdot)$ is Gateaux differentiable at $\varphi$ if and only if for any $\psi \in C(M, \mathbb{R})$
$$d^+P^u(f, \varphi)(\psi)=-d^+P^u(f, \varphi)(-\psi).$$
\begin{lemma}\label{derivative1}
For $\varphi, \psi\in C(M, \mathbb{R})$, $d^+P^u(f, \varphi)(\psi)=\sup\{\int_M \psi d\mu: \mu \in t_\varphi^u(M, f)\}.$
\end{lemma}
\begin{proof}
If $\mu \in t_\varphi^u(M, f)$, then $\int_M \psi d\mu \leq \frac{1}{t}(P^u(f, \varphi+t\psi)-P^u(f, \varphi))$ for $\forall t>0$. So $\int_M \psi d\mu \leq d^+P^u(f, \varphi)(\psi).$

Conversely, let $a=d^+P^u(f, \varphi)(\psi)$. Define a linear functional $\gamma$ on $\{t\psi: t\in \mathbb{R}\}$ by $\gamma(t\psi)=ta$. By the convexity of $P^u(f, \cdot)$, $\gamma(t\psi)=ta\leq P^u(f, \varphi+t\psi)-P^u(f, \varphi).$ Using Hahn-Banach theorem, $\gamma$ can be extended to a linear functional on $C(M, \mathbb{R})$ such that
$$\gamma (h)\leq P^u(f, \varphi+h)-P^u(f, \varphi).$$
By Riesz representation theorem, there exists $\mu \in t_\varphi^u(M, f)$ such that
$$\int_M \psi d\mu=\gamma(\psi)=a=d^+P^u(f, \varphi)(\psi).$$
\end{proof}

\begin{proof}[Proof of Theorem F]
If $P^u(f, \cdot)$ is Gateaux differentiable at $\varphi$, then for any $\psi \in C(M, \mathbb{R})$,
$$d^+P^u(f, \varphi)(\psi)=-d^+P^u(f, \varphi)(-\psi).$$
By Lemma \ref{derivative1}, we have for any $\psi \in C(M, \mathbb{R})$,
\begin{equation*}
\begin{aligned}
\sup\left\{\int_M \psi d\mu: \mu \in t_\varphi^u(M, f)\right\}&=-\sup\left\{\int_M (-\psi) d\mu: \mu \in t_\varphi^u(M, f)\right\}\\
&=\inf\left\{\int_M \psi d\mu: \mu \in t_\varphi^u(M, f)\right\}.
\end{aligned}
\end{equation*}
It follows that $t_\varphi^u(M, f)$ consists of a single element $\mu_\varphi$.

Conversely, suppose that $t_\varphi^u(M, f)$ consists of a single element $\mu_\varphi$. Then by Lemma \ref{derivative1}
$$d^+P^u(f, \varphi)(\psi)=\int_M \psi d\mu_\varphi=-\int_M (-\psi) d\mu_\varphi=-d^+P^u(f, \varphi)(-\psi).$$
So $P^u(f, \cdot)$ is Gateaux differentiable at $\varphi$.
\end{proof}
Corollary F.1 follows directly from Theorems E and F.

\subsection{Fr\'{e}chet differentiability}

\begin{lemma}\label{uniquetf}
$P^u(f, \cdot)$ has a unique u-tangent functional at $\varphi$ if and only if there is a unique measure $\mu_\varphi$ such that whenever $(\mu_n)\subset \mathcal{M}_f(M)$ with $h_{\mu_n}^u(f)+\int_M \varphi d\mu_n\to P^u(f, \varphi)$ we have $\mu_n \to \mu_\varphi$ as $n\to \infty.$
\end{lemma}
\begin{proof}
If $h_{\mu_n}^u(f)+\int_M \varphi d\mu_n\to P^u(f, \varphi)$ and $\mu_n \to \mu$ for some $\mu \in \mathcal{M}_f(M)$ as $n\to \infty$, by the upper semi-continuity of unstable entropy map, we have
$$h_{\mu}^u(f)+\int_M \varphi d\mu= P^u(f, \varphi).$$
Namely, $\mu\in \mathcal{M}^u_\varphi(M,f)=t_\varphi^u(M, f).$ Hence there is only one such measure $\mu$ which is denoted by $\mu_\varphi.$ This finishes the proof for the ``only if'' part.

For``if" part, let $\mu$ be a u-tangent functional to $P^u(f, \cdot)$ at $\varphi$. Note that such $\mu$ exists by Theorem E and Theorem B(2). Put $\mu_n =\mu$. Then
$h_{\mu_n}^u(f)+\int_M \varphi d\mu_n\to P^u(f, \varphi)$. Thus $\mu=\mu_\varphi$. It follows that $\mu_\varphi$ is the unique u-tangent functional of $P^u(f, \cdot)$ at $\varphi$.
\end{proof}

Recall that $P^u(f, \cdot)$ is called Fr\'{e}chet differentiable at $\varphi$ if $\exists \gamma \in C(M, \mathbb{R})^*$ such that
$$\lim_{\psi \to 0}\frac{|P^u(f, \varphi+\psi)-P^u(f, \varphi)-\gamma(\psi)|}{\|\psi\|}=0$$
(see Definition \ref{frechet}). If $P^u(f, \cdot)$ is Fr\'{e}chet differentiable at $\varphi$, then
$$\lim_{t\to 0}\frac{|P^u(f, \varphi+t\psi)-P^u(f, \varphi)-t\gamma(\psi)|}{t\|\psi\|}=0.$$
Thus $P^u(f, \cdot)$ is Gateaux differentiable at $\varphi$. Moreover, $\gamma(\psi)=\int_M \psi d\mu_\varphi$ where $\mu_\varphi$ is the unique u-tangent functional to $P^u(f, \cdot)$ at $\varphi$ by Theorem F and Lemma \ref{derivative1}.

\begin{proof}[Proof of Theorem G]
$(1) \Rightarrow (2)$: Suppose that $P^u(f, \cdot)$ is Fr\'{e}chet differentiable at $\varphi$. By Theorem F, let $\mu_\varphi$ be the unique u-tangent functional at $\varphi$. Let $(\mu_n)\subset \mathcal{M}_f(M)$ with $h_{\mu_n}^u(f)+\int_M \varphi d\mu_n\to P^u(f, \varphi)$ as $n\to \infty.$ Put $\epsilon_n:=P^u(f, \varphi)-h_{\mu_n}^u(f)-\int_M \varphi d\mu_n$. For any $\epsilon \in (0, \frac{1}{2})$, there exists $\delta>0$ such that whenever $\|\psi\|<\delta$ we have
$$0\leq P^u(f, \varphi+\psi)-P^u(f, \varphi)-\int_M \psi d\mu_\varphi\leq \epsilon \|\psi\|.$$
Then
\begin{equation*}
\begin{aligned}
\int_M \psi d\mu_n-\int_M \psi d\mu_\varphi &= P^u(f, \varphi)+\int_M \psi d\mu_n -P^u(f, \varphi)-\int_M \psi d\mu_\varphi\\
&=h_{\mu_n}^u(f)+\int_M \varphi d\mu_n+\epsilon_n+\int_M \psi d\mu_n -P^u(f, \varphi)-\int_M \psi d\mu_\varphi\\
&\leq P^u(f, \varphi+\psi)-P^u(f, \varphi)-\int_M \psi d\mu_\varphi+\epsilon_n\\
&\leq \epsilon\delta+\epsilon_n.
\end{aligned}
\end{equation*}
This is true for $-\psi$ too, and hence we have $|\int_M \psi d\mu_n-\int_M \psi d\mu_\varphi|\leq \epsilon\delta+\epsilon_n$ whenever $\|\psi\|\leq \delta.$  Thus
\begin{equation*}
\begin{aligned}
\|\mu_n-\mu_\varphi\| &=\sup\{ |\int_M \psi d\mu_n-\int_M \psi d\mu_\varphi|: \|\psi\|\leq 1\}\\
&=\frac{1}{\delta}\sup\{ |\int_M \psi d\mu_n-\int_M \psi d\mu_\varphi|: \|\psi\|\leq \delta\}\\
&\leq \epsilon+\frac{\epsilon_n}{\delta}< 2\epsilon \quad \text{\ for large enough\ }n.
\end{aligned}
\end{equation*}
So $\|\mu_n-\mu_\varphi\|\to 0$.

$(2) \Rightarrow (3)$: Assume $(2)$ holds. By Lemma \ref{uniquetf}, $\mu_\varphi$ is the unique member of $t^u_\varphi(M,f)$. By Theorem A, there exist $(\mu_n)$ ergodic with $h_{\mu_n}^u(f)+\int_M \varphi d\mu_n\to P^u(f, \varphi)$. By $(2)$, $\|\mu_n -\mu_\varphi\|\to 0$. Recall that distinct ergodic measures have norm-distance $2$. Thus there exists $N\in \mathbb{N}$ such that $\mu_n=\mu_\varphi$ for any $n\geq N$. This implies $P^u(f, \varphi)>\sup\{h_{\mu}^u(f)+\int_M \varphi d\mu: \mu \text{\ is ergodic and\ }\mu \neq \mu_\varphi\}.$

$(3) \Rightarrow (4)$: Assume $(3)$ holds. By Theorem E, $\mu_\varphi \in \mathcal{M}_\varphi(M,f)$. Put $a=P^u(f, \varphi)-\sup\{h_{\mu}^u(f)+\int_M \varphi d\mu: \mu \text{\ is ergodic and\ }\mu \neq \mu_\varphi\}.$ Define
$$V:=\{\mu \in \mathcal{M}_f(M): \Big|\int_M \varphi d\mu-\int_M \varphi d\mu_\varphi\Big|<a/2\}.$$
If $\mu \in V$ ergodic and $\mu\neq \mu_\varphi$, then
\begin{equation*}
\begin{aligned}
h_\mu^u(f)&\leq h_\mu^u(f)+\int_M \varphi d\mu-\int_M \varphi d\mu_\varphi+a/2\\
&\leq P^u(f, \varphi)-\int_M \varphi d\mu_\varphi-a/2\\
&=h_{\mu_\varphi}^u(f)-a/2.
\end{aligned}
\end{equation*}
Thus $(4)$ is proved.

$(4) \Rightarrow (3)$: Assume $(4)$ holds. Assume that $P^u(f, \varphi)=\sup\{h_{\mu}^u(f)+\int_M \varphi d\mu: \mu \text{\ is ergodic and\ }\mu \neq \mu_\varphi\}.$ Then there exist $\mu_n\neq \mu_\varphi$ ergodic such that $h_{\mu_n}^u(f)+\int_M \varphi d\mu_n \to P^u(f, \varphi)$. By Lemma \ref{uniquetf}, $\mu_n\to \mu_\varphi$ and hence $\mu_n\in V$ eventually. Thus $\limsup_{n\to \infty}h_{\mu_n}^u(f)+\int_M \varphi d\mu_\varphi=P^u(f, \varphi)$. This implies that $\limsup_{n\to \infty}h_{\mu_n}^u(f)=h_{\mu_\varphi}^u(f)$. A contradiction to $(4)$. So $(3)$ holds.

$(3) \Rightarrow (5)$: Assume $(3)$ holds. Put
$$a=P^u(f, \varphi)-\sup\Big\{h_{\mu}^u(f)+\int_M \varphi d\mu: \mu \text{\ is ergodic and\ }\mu \neq \mu_\varphi\Big\}.$$
If $\|\varphi-\psi\|<a/2$, then by Proposition \ref{proppressure} (4),
\begin{equation*}
\begin{aligned}
&\sup\Big\{h_{\mu}^u(f)+\int_M \psi d\mu: \mu \text{\ is ergodic and\ }\mu \neq \mu_\varphi\Big\}\\
\leq &P^u(f,\varphi)+\|\varphi-\psi\|-a\\
\leq &P^u(f, \psi)+2\|\varphi-\psi\|-a\\
<&P^u(f, \psi).
\end{aligned}
\end{equation*}
So all such $\psi$ have $\mu_\varphi$ as the unique u-equilibrium state by Theorem A. Thus $P^u(f, \psi)=h_{\mu_\varphi}^u(f)+\int_M \psi d\mu_\varphi$ for all $\psi$ in the ball centered at $\varphi$ of radius $a/2$.

$(5)\Rightarrow (6)\Rightarrow (7)$ is clear.

$(7) \Rightarrow (1)$: Assume $(7)$ holds. Let $\mu \in t^u_{\varphi+\psi}(M, f)=\mathcal{M}_{\varphi+\psi}^u(M,f)$ by Theorem E. Then $P^u(f, \varphi+\psi)=h_{\mu}^u(f)+\int_M (\psi+\varphi) d\mu$. We have
\begin{equation*}
\begin{aligned}
0&\leq P^u(f, \varphi+\psi)-P^u(f, \varphi)-\int_M \psi d\mu_\varphi\\
&\leq h_{\mu}^u(f)+\int_M (\psi+\varphi) d\mu-h_{\mu}^u(f)-\int_M \varphi d\mu-\int_M \psi d\mu_\varphi\\
&=\int_M \psi d\mu-\int_M \psi d\mu_\varphi\\
&\leq \|\psi\|\cdot \|\mu-\mu_\varphi\|.
\end{aligned}
\end{equation*}
Hence $0\leq P^u(f, \varphi+\psi)-P^u(f, \varphi)-\int_M \psi d\mu_\varphi \leq \|\psi\|\inf\{\|\mu-\mu_\varphi\|: \mu \in t^u_{\varphi+\psi}(M, f)\}$. Hence $P^u(f, \cdot)$ is Fr\'{e}chet differentiable at $\varphi$.
\end{proof}

\end{document}